\documentclass[11pt]{article}
\usepackage[margin =1in]{geometry}
\usepackage{amssymb,amsthm,amsmath}
\usepackage{enumerate}
\usepackage{hyperref}
\usepackage{cite}
\usepackage[capitalize]{cleveref}
\usepackage{enumitem}
\usepackage{color}
\usepackage{cancel}
\usepackage{pgf}
\usepackage{svg}
\usepackage{pgfplots}
\usepackage{import}

\usepackage[utf8]{inputenc} 
\usepackage[T1]{fontenc}    
\usepackage{hyperref}       
\usepackage{url}            
\usepackage{booktabs}       
\usepackage{amsfonts}       
\usepackage{nicefrac}       
\usepackage{microtype}
\usepackage{multirow}
\usepackage{xfrac}
\usepackage{todonotes}
\usepackage{titlesec}
\usepackage{caption}
\usepackage{subcaption}
\usepackage{algorithm}
\usepackage{algorithmic}
\usepackage{tcolorbox}
\tcbset{
  myboxstyle/.style={
    colback=blue!3!white,       
    colframe=blue!50!black,     
    boxrule=1pt,              
    arc=3pt,                    
    left=6pt, right=6pt, top=4pt, bottom=4pt,
  }
}

\titleformat{\subsubsection}[runin]
{\normalfont\normalsize\bfseries}{\thesubsubsection}{1em}{}

\usepackage{graphicx}

\usepackage{appendix}
\numberwithin{equation}{section}

\newcommand{\argmin}{\mathrm{argmin}}

\newtheorem{thm}{Theorem}[section]
\newtheorem{theorem}{Theorem}[section]

\newtheorem{lemma}[thm]{Lemma}

\crefname{claim}{claim}{claims}
\Crefname{claim}{Claim}{Claims}
\crefname{lem}{lemma}{lemmas}
\Crefname{lem}{Lemma}{Lemmas}
\crefname{algorithm}{algorithm}{algorithms}
\Crefname{algorithm}{Algorithm}{Algorithms}

\theoremstyle{remark}

\allowdisplaybreaks

\definecolor{blue}{rgb}{0,0,0}
\definecolor{red}{rgb}{0,0,0}

\begin{document}

	\title{Lower Bounds for Linear Minimization Oracle Methods\\ Optimizing over Strongly Convex Sets}

	  \author{Benjamin Grimmer\footnote{Johns Hopkins University, Department of Applied Mathematics and Statistics, \url{grimmer@jhu.edu}} \qquad  Ning Liu\footnote{Johns Hopkins University, Department of Applied Mathematics and Statistics, \url{nliu15@jhu.edu}}}

	\date{}
	\maketitle

	\begin{abstract}
    We consider the oracle complexity of constrained convex optimization given access to a Linear Minimization Oracle (LMO) for the constraint set and a gradient oracle for the $L$-smooth, $L$-strongly convex objective. This model includes Frank-Wolfe methods and their many variants. Over the problem class of $\alpha$-strongly convex constraint sets $S$, {\color{red} we demonstrate that one can construct hard ``zero-chain'' instances in the classical style of Nemirovski and Yudin. From our new approach to adversarial oracle construction, we} prove that no such deterministic method can guarantee a final objective gap less than $\varepsilon$ in fewer than $\Omega(\sqrt{L\, \mathrm{diam}(S)^2/\varepsilon})$ iterations. Our lower bound {\color{red} partly} matches the accelerated Frank-Wolfe theory of Garber and Hazan~\cite{garber2015faster} of $O(\sqrt{L(\mathrm{diam}(S)^2{\color{red}+1/\alpha^2})/\varepsilon})$. Second, we consider optimization over $\beta$-smooth sets, finding that in the modestly smooth regime of $\beta=\Omega(1/\sqrt{\varepsilon})$, no complexity improvement for span-based LMO methods is possible against either compact convex sets or strongly convex sets.
\end{abstract}

    \section{Introduction}
In this work, we consider convex constrained optimization problems of the form
\begin{equation} \label{eq:main}
    \begin{cases}
        \min & f(x)\\
        \mathrm{s.t.\ } & x\in S \subseteq \mathbb{R}^d.
    \end{cases}    
\end{equation}
In particular, we consider high-dimensional problems where $d$ may be arbitrarily large.
Frank-Wolfe methods and the broader family of ``projection-free'' algorithms using a linear minimization subroutine have found renewed interest due to their scalability. See the survey~\cite{braun2025conditionalgradientmethods} and references therein. In a basic form, these methods iterate from an initialization $x_0$, producing for $k=0,1,\dots$
\begin{align*}    
    p_k &= \nabla f(x_k),\\
    z_{k+1} & \in\argmin_{x\in S}\langle p_k, x\rangle,\\
    x_{k+1} & \in \operatorname{conv}\{x_k, z_{k+1}\}.
\end{align*}
As examples, an exact line search implementation of Frank-Wolfe would set $x_{k+1}$ as the minimizer of $f$ on the segment $[x_k,z_{k+1}]$. Similarly, a fixed ``open-loop'' stepsize implementation would fix $x_{k+1} = \theta_k x_k + (1-\theta_k)z_{k+1}$ for predetermined $\theta_k\in [0,1]$.
The two key computational oracles assumed here are access to gradients of the objective $f$ to compute $p_k$ and access to a Linear Minimization Oracle (LMO) to produce $z_{k+1}$. As a shorthand, we denote $\mathtt{LMO}_S(p) \in \argmin_{x\in S}\langle p, x\rangle$ as an oracle providing a selection of this minimizer (potentially selected adversarially).

In this work, we provide complexity lower bounds {\color{blue}for families} of first-order methods using a linear minimization oracle. {\color{blue} Our theory covers any deterministic method and any methods remaining in the span of observed gradients and LMO solutions, containing the above Frank-Wolfe methods.} We focus on lower bounds for problem classes where the constraint set $S$ possesses structural properties like strong convexity or smoothness. Below, we formalize these classes of problems and algorithms.

\paragraph{The Families of Smooth and Strongly Convex Constraint Sets and Functions} {\color{blue} Here we consider} problem instances defined by a differentiable function $f\colon \mathbb{R}^d\rightarrow \mathbb{R}$ and a set $S\subseteq\mathbb{R}^d$ as well as a {\color{blue} feasible} initialization $x_0{\color{blue} \in S}$. {\color{blue} Noting the translation invariance of the optimization problem~\eqref{eq:main}, the choice of $x_0$ can be fixed as an arbitrary feasible point. Our analysis will often fix $x_0 = 0$ and consider constraint sets $S$ containing the origin without loss of generality. One could consider algorithms with an arbitrary (not necessarily feasible) initialization. By developing lower bounds for problems with feasible $x_0$, we immediately provide bounds on this larger class.}

A problem class is defined by a set of allowable $f$ and $S$ values. We consider the standard family of convex objective functions parameterized by $0\leq\mu\leq L$, defined as
\begin{align*}
    \text{$f$ is $\mu$-strongly convex if\ }& f(\lambda x + (1-\lambda)y)\leq \lambda f(x) + (1-\lambda)f(y) - \frac{\lambda(1-\lambda)\mu}{2}\|y-x\|_2^2\\
    &\hskip5.6cm \forall x,y\in \mathbb{R}^d, \lambda\in[0,1], \\
    \text{$f$ is $L$-smooth if\ } & \|\nabla f(y)-\nabla f(x)\|_2\leq L\|y-x\|_2 \quad \forall x,y\in \mathbb{R}^d.
\end{align*}
If $L=\mu$, then $f$ must be a quadratic function of the form $\frac{L}{2}\|x-x_\star\|_2^2$ up to an additive constant. Such simple quadratics will suffice for our theoretical development.

We consider compact convex sets $S$ with diameter $\mathrm{diam}(S)=\max_{x,y\in S}\|x-y\|_2$ with at least one of the natural parallel notions of strong convexity and smoothness for constraint sets. We say
$$ \text{$S$ is $\alpha$-strongly convex if\quad } \lambda x + (1-\lambda)y + B\left(0,\ \frac{\lambda(1-\lambda)\alpha}{2}\|y-x\|_2^2\right) \subseteq S $$
for all $x,y\in S, \lambda\in [0,1]$ where $B(0,r)=\{x\in\mathbb{R}^d \mid \|x\|_2\leq r\}$ denotes the closed ball of radius $r$. Normal vectors $n \in N_S(x) := \{n \mid \langle n, y-x\rangle \leq 0\ \forall y\in S\}$ at each $x\in\operatorname{bdry}S$ serve as the analogues of gradients. Considering any unit length normal vectors, we define
\begin{align*}
    \text{$S$ is $\beta$-smooth if\quad } \|n_y - n_x\|_2 \leq \beta \|y-x\|_2 \quad \forall &x,y\in\operatorname{bdry}S,\\
    &n_y\in \{n\in N_S(y) \mid \|n\|_2=1\},\\
    &n_x\in \{n\in N_S(x) \mid \|n\|_2=1\}.
\end{align*}
Note that the above definition requires that smooth sets have a unique unit normal vector at each boundary point.
For a classical reference providing equivalent characterizations of the strong convexity of a set, see~\cite{vial1982strong}. A modernized treatment of smoothness and strong convexity of sets was given by~\cite{liu2023gauges}. Therein, parallels to the smoothness and strong convexity of functions are explored.\\

\paragraph{The Family of First-Order Linear Minimization Oracle Methods (\texttt{FO-LMO})}
For a given problem instance $(f,S)$ and initialization $x_0 \in S$, a \texttt{FO-LMO} method generates sequences of search directions $p_k\in\mathbb{R}^d$, linear minimization solutions $z_{k+1} = \mathtt{LMO}_S(p_k)\in\argmin_{x\in S}\langle p_k, x\rangle$, and iterates $x_{k+1}$. We require that $p_k$ is a deterministic function of the oracle responses $\{({\color{blue} x_\ell},f(x_\ell),\nabla f(x_\ell))\}^k_{\ell=0}$ and $\{z_{\ell}\}^k_{\ell=1}$ so far and that $x_{k+1}$ is a deterministic function of the responses $\{({\color{blue} x_\ell}, f(x_\ell),\nabla f(x_\ell))\}^k_{\ell=0}$ and $\{z_{\ell}\}^{k+1}_{\ell=1}$. Note that we do not require that $x_{k+1}$ lie in the convex hull of $x_0$ and $\{z_{\ell}\}_{\ell=1}^{k+1}$. However, when outside this convex hull, the iterates may fail to be feasible.

{\color{blue} Note that the \texttt{FO-LMO} model does not include linesearch methods: Determining a minimizer on a line or subspace is not a deterministic function of past first-order observations. In Section~\ref{sec:SC}, we consider an alternative algorithm model, dubbed \texttt{LMO-span} methods, allowed to select each iterate anywhere within a span associated with the observed values so far, enabling linesearches. Together, these two families of methods include}, for example, the Away-step Frank-Wolfe methods~\cite{lacoste2015global,beck2017linearly}, Pairwise Frank-Wolfe methods~\cite{lacoste2015global,tsuji2022pairwise}, Fully Corrective Frank-Wolfe methods~\cite{holloway1974extension,lacoste2015global}, and Blended Pairwise Frank-Wolfe methods~\cite{braun2019blended}.

Each iteration of these algorithms can make one call to a first-order oracle, returning $(f(x_k), \nabla f(x_k))$, and one call to the LMO. Our theory then bounds iteration complexity to measure the minimum number of such pairs of oracle calls needed to {\color{blue} produce a feasible point with $\varepsilon$-suboptimality (i.e., $x_T\in S$ and $f(x_T) - \min_{x\in S}f(x) \leq \varepsilon$)}. \\


For the minimization of an $L$-smooth convex function $f$ over a compact convex set $S$, Frank-Wolfe methods~\cite{frank1956algorithm,jaggi2013revisiting} are known to provide convergence rates, {\color{blue} for any $T>0$, at $x_T\in S$}
$$ f(x_T) - \min_{x\in S}f(x) \leq \mathcal{O}\left(\frac{L\, \mathrm{diam}(S)^2}{T}\right). $$
This provides an upper bound on the iteration complexity of computing a {\color{blue} feasible} point with $\varepsilon$-suboptimality (i.e., $f(x_T)-\min_{x\in S}f(x)\leq\varepsilon$) of $\mathcal{O}(L\, \mathrm{diam}(S)^2/\varepsilon)$. A matching complexity lower bound was provided by Lan~\cite{lan2013complexity}, establishing this as the order of the optimal LMO complexity.
Garber and Hazan~\cite{garber2015faster} showed that given the additional structure that $f$ is $\mu$-strongly convex and $S$ is $\alpha$-strongly convex, this rate can be accelerated to have
\begin{equation} \label{eq:garber-hazan-rate}
    f(x_T) - \min_{x\in S}f(x) \leq \mathcal{O}\left(\frac{L(\mathrm{diam}(S)^2 {\color{red} + L/(\mu\alpha^2)})}{T^2}\right)
\end{equation}
{\color{blue} and $x_T\in S$ for any $T>0$,} giving a $\mathcal{O}(\sqrt{L(\mathrm{diam}(S)^2{\color{red}+L/(\mu\alpha^2)})/\varepsilon})$ iteration complexity. This accelerated result poses the natural question of whether further acceleration is possible or if this complexity is the optimal strongly convex order.

Providing recent progress on this question, Halbey et al.~\cite{halbey2026lower} showed that two particular variants of Frank-Wolfe (exact line search and the short stepsize procedure) cannot improve upon this $\mathcal{O}(1/T^2)$ rate. However, these results do not preclude the possibility of other \texttt{FO-LMO} methods exceeding this rate.
Similar limitations in unconstrained minimization were overcome by the foundational work~\cite{nemirovski1983problem}, establishing information/oracle complexity lower bounds against the whole family of gradient methods. This was accomplished by the design of hard functions whose gradient reveals only one new coordinate of information per step. This property, known as a ``zero-chain'' property, can prevent {\it any gradient-span method} from having made substantial progress until dimension-many steps have been taken. By combining {\color{blue} such a hard instance} with an adversarial ``resisting oracle'', {\color{blue} a family of hard instances can be constructed, providing} lower bounds against {\it all deterministic gradient methods}.\\

\paragraph{Our Contributions} We extend the classical zero-chain approach to LMOs to derive complexity lower bounds on any \texttt{FO-LMO} method over the problem classes corresponding to minimizing a quadratic $f(x)=\frac{L}{2}\|x-x_\star\|_2^2$ over structured sets, possessing either strong convexity or smoothness. From such a construction, we show that over $\alpha$-strongly convex constraints, no \texttt{FO-LMO} method can guarantee $\varepsilon$-suboptimality {\color{blue} at a feasible iterate} in fewer than $\Omega(\sqrt{L\, \mathrm{diam}(S)^2/\varepsilon})$ iterations. {\color{red} Although this does not capture the second term in the rate of}~\cite{garber2015faster} {\color{red} dependent on $\alpha$, this provides progress towards determining the exact} optimal order of complexity for this class. Formally, we prove the following in Section~\ref{sec:general}.

\begin{theorem}\label{thm:main}
Consider any $T \ge 1$ and $L,\alpha>0$. Then {\color{blue} for every \texttt{FO-LMO},} there exist an $\alpha$-strongly convex set $S$ and an $L$-smooth, $L$-strongly convex function $f$ in dimension $d=2(T+1)$ such that the  method applied to $(f, S)$ with $x_0 = 0$ has $x_T$ infeasible or 
\[
f(x_T) - \min_{x\in S}f(x) \geq \frac{1}{528}\frac{L\, \mathrm{diam}(S)^2}{(T+1)^2}.
\]
Consequently, for any $\varepsilon>0$, there exist problem instances where $T\geq \sqrt{\frac{L\, \mathrm{diam}(S)^2}{528\varepsilon}} - 1$ iterations are required to {\color{blue} find a feasible $x_T$ with} $f(x_T) - \min_{x\in S}f(x) \leq \varepsilon$.
\end{theorem}
While Theorem~\ref{thm:main} applies for any $L,\alpha>0$, it does not allow a free selection of $\mathrm{diam}(S)$. Rather, our constructed hard ``zero-chain'' sets $S$ have $\mathrm{diam}(S) = \Theta(1/(\alpha d))$. {\color{red} Applying this diameter bound and $d=\Theta(T)$, our lower bound in terms of $L$ and $\alpha$ becomes $\Omega(L/(\alpha^2T^4)$, differing from the second term in~\eqref{eq:garber-hazan-rate} by $1/T^2$. While not tight, this limits the potential for linearly convergent methods.}

Note that one cannot arbitrarily select $\alpha$ and $\mathrm{diam}(S)$. They must satisfy, for example, that $\mathrm{diam}(S)\leq 2/\alpha$. 
In the limit where $\mathrm{diam}(S)= 2/\alpha$, the set $S$ is forced to be (up to translation) the ball $B(0,1/\alpha)$. {\color{blue} This can be verified using the fact that every $\alpha$-strongly convex set $S$ equals an intersection of (infinitely many) balls of radius $1/\alpha$~\cite[Theorem 1]{vial1982strong},
\begin{equation*}
    S = \bigcap_i B(c_i, 1/\alpha).
\end{equation*}
Observe that the only ball of radius $1/\alpha$ that contains a pair of points $x,y\in S$ attaining the claimed diameter bound $\|x-y\|=2/\alpha$ is the ball centered at $c=(x+y)/2$. Hence every ball in the above formula must be $B(c,1/\alpha)$ and so $S=B(c,1/\alpha)$.}

Such forced structure enables faster algorithms: {\color{blue} For example, when $\mathrm{diam}(S)= 2/\alpha$ and $S = B(0,1/\alpha)$, an LMO can be used to explicitly compute orthogonal projections onto the feasible region $S$. Then the class of \texttt{FO-LMO} methods includes projected gradient methods, which are known to converge linearly for smooth, strongly convex objectives~\cite[Theorem 3.10]{Bubeck2015}.} The development of hard instances for any selection of $\mathrm{diam}(S) \leq 2/\alpha$ and resulting more nuanced complexity bounds is left as an important future direction.

As a secondary result, we provide a partial generalization to $\beta$-smooth sets. These bounds are meaningful in the regime of only modestly smooth sets, having $\beta=\Omega(1/\sqrt{\varepsilon})$. In this regime, no \texttt{LMO-span} method (see Section~\ref{sec:SC} for a formal definition) can improve past the optimal compact convex set complexity of $\Theta(L\, \mathrm{diam}(S)^2/\varepsilon)$ or our $\alpha$-strongly convex set lower bound of $\Omega(\sqrt{L\, \mathrm{diam}(S)^2/\varepsilon})$. Theorems~\ref{thm:smooth} and~\ref{thm:mixed} formalize these limits on acceleration due to smoothness.

Since our constructions throughout use an $L$-smooth, $L$-strongly convex objective, all of our lower bounds apply against the wider class of problems with smooth, strongly convex objectives having any $0 \leq \mu \leq L$. Hence, our theory highlights a fundamental difficulty of constrained optimization via LMOs: {\it Even on perfectly conditioned objective functions, hard adversarial constraint sets constitute a barrier to linear convergence.} This stands in contrast to gradient methods with an orthogonal projection oracle where convergence is dominated by objective function conditioning.

\paragraph{Outline} Section~\ref{sec:SC} first derives a lower bound via a novel construction of a strongly convex feasible region that is hard for all \texttt{LMO-span} methods. Section~\ref{sec:general} {\color{blue}then provides a hard instance for each} \texttt{FO-LMO} method, proving Theorem~\ref{thm:main}. This construction is the main technical innovation of our work. Section~\ref{sec:smooth} then provides results extending lower bounds to the specialized setting of sets with modest levels of smoothness.

\subsection{Related Work}
Our Theorem~\ref{thm:main} differs from the previously mentioned lower bounding result of~\cite{halbey2026lower} in two aspects, making the result complementary. In terms of algorithmic scope, our bound provides a wider guarantee, establishing a universal lower bound against all \texttt{FO-LMO} methods. In contrast,~\cite{halbey2026lower} provides a hard instance for two standard implementations of Frank-Wolfe, namely those fixing $p_k=\nabla f(x_k)$ and using an exact line search or the short stepsize procedure to select $x_{k+1}$. In terms of problem class scope, our construction requires a nonsmooth feasible region and a large problem dimension $d$ (linear in the number of iterations to be run). Such a high dimensionality assumption (i.e., $d>T$) is classical and widespread in the optimization complexity literature~\cite{nemirovski1983problem}. In contrast,~\cite{halbey2026lower} is able to provide a hard instance using a smooth ball in $d=2$ dimensions. As a result, they provide a stronger illustration of the limitations of the two methods that their theory covers.

As mentioned above, most classical lower bounding results in unconstrained first-order optimization rely on setting the problem dimension larger than the number of iterations to be conducted. This enables ``zero-chain'' arguments where the objective function is designed to reveal only one new coordinate to the given gradient-span algorithm at each iteration. The main technical innovation of our work is the design of a hard strongly convex set where the LMO possesses a similar zero-chain property to these classical hard objective constructions. Such constructions open the possibility to extend proof techniques and insights from existing unconstrained optimization lower bounds to constrained LMO/projection-free settings.

For example, {\color{red} comparing our lower bound in terms of $L$ and $\mathrm{diam(S)}$ with the first term of~\eqref{eq:garber-hazan-rate}}, the two differ by constant factors: our lower bound has a universal constant of $1/528$ and the upper bound has $9/2$. Determining exactly minimax optimal algorithms and hard problem instances is an important future direction. In settings of unconstrained first-order minimization, the Performance Estimation Problem (PEP) techniques pioneered by~\cite{drori2014performance,Interpolation,Interpolation2} have provided such theory. For example, in smooth convex optimization, these facilitated the identification of the Optimized Gradient Method~\cite{kim2016optimized} and an exactly matching lower bound~\cite{drori2017exact}. PEP was extended to cover structured smooth and strongly convex sets by Luner and Grimmer~\cite{luner2024performance}. As a result, future work may leverage PEP to similarly tighten gaps left here.

Another important property of most Frank-Wolfe methods is that their trajectory is independent of the choice of inner product used. This differs from projected-gradient methods, where the choice of inner product and notion of orthogonality for projections affect the algorithm trajectory, making good preconditioning important for practical success. The line of work~\cite{Pena2023,Wirth2025,Wirth2026} provided Frank-Wolfe with ``affine-covariant'' convergence theory, matching the method's affine covariant nature by avoiding notions like smoothness and strong convexity defined in terms of a fixed inner product. Developing lower bounding theory in affine-covariant terms is an interesting future direction.

One can view the oracle model of an LMO as assuming a first-order oracle for the support function{\color{blue}
$$\sigma_S(p) = \sup\{\langle p, x\rangle \mid x\in S\}.$$ 
That is, the subdifferential of $\sigma_S$ at $p$ is the set of maximizers of $\langle p,x\rangle$ over $S$ (so the possible LMO solutions are subgradients of $\sigma_S$ at $-p$).
For sets with $0\in S$, the support function has a dual relationship to the Minkowski gauge
$$\gamma_S(y) = \inf\{ \gamma>0 \mid y/\gamma \in S\}.$$
Namely, denoting the polar of $S$ by $S^\circ = \{p \mid \langle p,x\rangle \leq 1 \ \forall x\in S\}$, $\sigma_S(p) = \gamma_{S^\circ}(p)$. Note that our definition of $\sigma_S$ may be infinite-valued if $S$ is not compact, and $\gamma_S$ may be infinite-valued if $0$ lies on the boundary of $S$.

Hence, dual to LMO methods, one may consider methods assuming gauge oracle access,} previously studied by~\cite{Renegar2016,Grimmer2021-part1,Grimmer2021-part2,Zakaria2022,Lu2023} as an alternative projection-free framework. The works~\cite{liu2023gauges} and~\cite{samakhoana2024scalable} developed accelerated $\mathcal{O}(1/T^2)$ convergence guarantees for gauge methods over smooth sets, complementing the accelerated LMO rates for strongly convex sets. Identifying any structural relationships between the complexity of these dual oracle models and problem settings is another interesting direction.

\section{Lower Bounds for Strongly Convex Sets and Span Methods}\label{sec:SC}
In this section, we develop the core ideas and constructions underlying our lower bounding theory. We do this against a modified family of \texttt{FO-LMO} methods, called \texttt{LMO-span} methods, restricted to select search directions and iterates from fixed spans and convex hulls. Using a ``resisting oracle'', in the next section, our constructions here will extend to lower bounds for any \texttt{FO-LMO} method.

\paragraph{A Family of First-Order Linear Minimization Oracle Span Methods (\texttt{LMO-span})}
For a given problem instance $(f,S)$ and initialization $x_0 \in S$, a \texttt{LMO-span} method generates sequences of search directions $p_k$, linear minimization solutions $z_{k+1}$, and iterates $x_{k+1}$ as follows for $k=0,1,\dots$
\begin{align}
    p_k &\in \operatorname{span}\left\{\{x_\ell-x_0\}_{\ell=1}^k \cup \{z_\ell-x_0\}_{\ell=1}^k \cup \{\nabla f(x_\ell)\}_{\ell=0}^k\right\} \setminus \{0\}, \label{eq:p-def}\\
    z_{k+1} & = \mathtt{LMO}_S(p_k)\in\argmin_{x\in S}\langle p_k, x\rangle, \label{eq:z-def}\\
    x_{k+1} &\in \operatorname{conv}\{x_0, z_1, \dots, z_{k+1}\}. \label{eq:x-def}
\end{align}
No computational restrictions are placed on how $p_k$ and $x_{k+1}$ are computed. So operations like exact line searches and any usage of a memory/bundle of past gradients are allowed within this model.
We also remark that the restriction that $p_k\neq 0$ is for ease of our development and without loss of generality. An adversarial LMO given $p_k=0$ can return any $z_{k+1}\in S$ as a minimizer. In particular, it could return the prior iterate $x_k$, providing the algorithm with no new information.

{\color{blue} Note that above, $p_k,x_k,z_k$ denote vectors in $\mathbb{R}^d$ with the subscripts indexing their place in a sequence. When doing such indexing, we will use indices $k,t,T$. At times, we will need to refer to the coordinates of such a vector. We will denote the coordinates of a vector by $(p)_i$, wrapped in parentheses. We will reserve indices $i,j$ for denoting coordinates.}

The following theorem provides universal lower bounds against any such span method applied to a strongly convex problem in the high-dimensional regime where $d>T$.
\begin{theorem}\label{thm:SC}
For any $d \ge 1$ and $L,\alpha>0$, there exist an $\alpha$-strongly convex set $S$ and an $L$-smooth, $L$-strongly convex function $f$ such that every \texttt{LMO-span} method applied to $(f, S)$ with $x_0 = 0$ has
\begin{equation}\label{eq:finalLB}
f(x_t) - \min_{x\in S}f(x) \geq \frac{2}{5}\frac{(d-t)L\, \mathrm{diam}(S)^2}{(d+2)^3} \qquad \forall t \in \{0, \dots, d-1\}.
\end{equation}
In particular, for any fixed budget $T \ge 1$, there exist $S$ and $f$ {\color{blue}satisfying the same conditions} in dimension $d = 2(T+1)$ such that
\[
f(x_T) - \min_{x\in S}f(x) \geq \frac{1}{20} \frac{L\, \mathrm{diam}(S)^2}{(T+2)^2}.
\]
Consequently, for any $\varepsilon>0$, there exist problem instances where $T \geq \sqrt{\frac{L\, \mathrm{diam}(S)^2}{20\varepsilon}} - 2$ iterations are required for any \texttt{LMO-span} method to reach a suboptimality of $f(x_T) - \min_{x\in S}f(x) \leq \varepsilon$.
\end{theorem}
{\color{red} This matches the first term in~\eqref{eq:garber-hazan-rate} but lacks a dependence on the curvature $\alpha$. Noting that our hard instance construction has $\mathrm{diam}(S)=\Theta(1/(\alpha d))$, one can rewrite our lower bound as $\Omega(L/(\alpha^2T^4))$, matching the second term up to a factor of $1/T^2$.}

Our proof of this result is developed in {\color{blue}four} parts. {\color{blue} First, we establish by simple rescaling arguments that it suffices to consider only the case of $L=\alpha=1$. Then} Section~\ref{subsec:Construction-SC} constructs our candidate hard problem instance for each dimension and verifies its validity (computing its strong convexity constant and diameter). Next Section~\ref{subsec:zero-chain-SC} establishes a key ``zero-chain'' property of these hard instances, showing that any \texttt{LMO-span} method applied will only discover one new coordinate per iteration. Finally, by leveraging this property, Section~\ref{subsec:proof-SC} proves this section's main result, Theorem~\ref{thm:SC}, showing that no method in $T$ steps can guarantee a suboptimality less than $\Omega(L\, \mathrm{diam}(S)^2/T^2)$. 

{\color{blue}
The following lemma shows that it suffices to fix $L=1$ and $\alpha=1$ throughout. This applies whether one considers \texttt{LMO-span} methods or \texttt{FO-LMO} methods.
\begin{lemma} \label{lem:rescaling}
    Suppose there exists a $1$-smooth, $1$-strongly convex function $f$ and a $1$-strongly convex set $S$ and a constant $r>0$ such that every \texttt{LMO-span} method (or \texttt{FO-LMO} method) has
    $$ f(x_T) - \min_{x\in S}f(x) \geq r\, \mathrm{diam}(S)^2.$$
    Then, for any $L,\alpha>0$, the rescaled instance defined by
    \begin{equation}\label{eq:rescalings}
         \tilde{f}(x) := \frac{L}{\alpha^2}f(\alpha x),\qquad \tilde{S} := \frac{1}{\alpha} S = \{x/\alpha  \mid x \in S\}.
    \end{equation}
    has $\tilde f$ being $L$-smooth and $L$-strongly convex, $\tilde S$ being $\alpha$-strongly convex, and every \texttt{LMO-span} method (or \texttt{FO-LMO} method) satisfies
    $$ \tilde{f}(x_T) - \min_{x\in\tilde{S}}\tilde{f}(x) \geq Lr\, \mathrm{diam}(\tilde{S})^2.$$
\end{lemma}
\begin{proof}
    Observe that the gradient of the rescaled function is given by $\nabla \tilde f(x) = (L/\alpha) \nabla f(\alpha x)$. Then $L$-smoothness follows from $1$-smoothness of $f$ as
    $$ \|\nabla \tilde{f}(x) - \nabla \tilde{f}(y)\| = \frac{L}{\alpha}\|\nabla f(\alpha x) - \nabla f(\alpha y)\| \leq \frac{L}{\alpha}\|\alpha x - \alpha y\| = L\|x-y\|.$$
    Similarly, $L$-strong convexity of $\tilde{f}$ follows from $1$-strong convexity of $f$ as
    \begin{align*}
        \tilde{f}(\lambda x + (1-\lambda)y)&= \frac{L}{\alpha^2} f(\lambda \alpha x + (1-\lambda)\alpha y)\\
        &\leq \frac{L}{\alpha^2}\left(\lambda f(\alpha x) + (1-\lambda)f(\alpha y) - \frac{\lambda(1-\lambda)}{2}\|\alpha y-\alpha x\|_2^2\right)\\
        &=\lambda \tilde{f}(x) + (1-\lambda)\tilde{f}(y) - \frac{\lambda(1-\lambda)L}{2}\|y-x\|_2^2.
    \end{align*}
    The rescaling of the given $1$-strongly convex set $S$ by $1/\alpha$ directly makes $\tilde{S}$ be $\alpha$-strongly convex and have $\mathrm{diam}(\tilde{S}) = \mathrm{diam}(S)/\alpha$.

    Now consider any \texttt{LMO-span} method, generating search directions $\tilde{p}_0,\dots,\tilde{p}_{T-1}$, minimizers $\tilde{z}_1,\dots,\tilde{z}_T$, and iterates $\tilde{x}_0,\dots,\tilde{x}_T$ when applied to $\tilde{f}$ and $\tilde{S}$. Consider the rescaled problem data 
    $p_k = (L/\alpha)\tilde{p}_k,\ z_k=\alpha \tilde{z}_k,\ x_k = \alpha \tilde{x}_k$. Since each $\nabla f(x_k)$ is a positive rescaling of $\nabla \tilde{f}(\tilde{x}_k)$, a direct recursive argument establishes $p_k$ must lie in~\eqref{eq:p-def}, $z_{k+1}$ must be the linear minimization solution for $p_k$ over $S$~\eqref{eq:z-def}, and $x_{k+1}$ satisfies~\eqref{eq:x-def}. So $x_T$ is the final iterate of some \texttt{LMO-span} method applied to $f$ and $S$. Hence $f(x_T) - \min_{x\in S}f(x) \geq r\, \mathrm{diam}(S)^2$. Equivalently, $\tilde{f}(\tilde{x}_T) - \min_{x\in\tilde{S}}\tilde{f}(x) \geq Lr\, \mathrm{diam}(\tilde{S})^2$.

    Likewise, for any \texttt{FO-LMO} method, the rescaled data remains a deterministic function of prior observations. So $x_T$ is the final iterate of some \texttt{FO-LMO} method, leading to the same conclusion.
\end{proof}
}

\subsection{Construction of a Hard Problem Instance}\label{subsec:Construction-SC}

{\color{blue} Recall that we denote the $i$th coordinate of a vector $x\in\mathbb{R}^d$ by $(x)_i$. Then} for any given problem dimension $d$, we construct our hard problem instance as follows:
The feasible region is defined by
\begin{equation} \label{eq:S-simple}
S := \left\{ x \in \mathbb{R}^d \mid \frac{1}{2}\|x\|_2^2 + \sum_{i=1}^d (w)_i |(x)_i| \le C^2 \right\}
\end{equation}
where the normalization constant $C$ and a vector of weights $w$ are defined as
\begin{equation}\label{eq:params}
C := \frac{1}{\sqrt{d^2+d+2}}, \qquad
(w)_i := C\sqrt{2 i} \quad \text{for } i=1,\dots,d.
\end{equation}
Note that $x_0=0$ is feasible, lying in the interior of $S$. The objective function is constructed to have its minimizer $x_\star = \nu \mathbf{1}$ where {\color{blue} $\mathbf{1}$ denotes the all ones vector in $\mathbb{R}^d$ and} $\nu$ is the unique positive root of
\begin{equation}\label{eq:nu-def}
\frac{1}{2} d \nu^2 + \left(\sum_{i=1}^d (w)_i\right)\nu  = C^2.
\end{equation}
This choice guarantees that $x_\star$ lies on the boundary of $S$ and hence is feasible. We then set $f(x):=\frac{1}{2}\|x-x_\star\|_2^2$. {\color{blue} This function is $1$-smooth, $1$-strongly convex, and minimizes at $x_\star$,} ensuring that the minimum of $f$ over $S$ is zero.

In the remainder of this subsection, we verify that $S$ has strong convexity constant $\alpha=1$ in Lemma~\ref{lem:SC} and compute its diameter in Lemma~\ref{lem:diam-SC}. Together, these establish the validity of the considered problem instance defined by $f$ and $S$.

\begin{lemma}[Strong Convexity of $S$] \label{lem:SC}
    The constructed set $S$ in~\eqref{eq:S-simple} is $\alpha=1$-strongly convex.    
\end{lemma}
\begin{proof}
    We prove this by using the fact that strong convexity is preserved under intersections~\cite[Proposition 2]{vial1982strong}. Hence, it suffices to show that $S$ is equal to an intersection of $1$-strongly convex sets. We do this below by reformulating $S$ into the intersection of $2^d$ shifted unit balls $\left\{ x \in \mathbb{R}^d \mid  \left\| x + s \circ w \right\|_2 \le 1 \right\}$ where $s \circ w = ((s)_1 (w)_1, \dots, (s)_d (w)_d)$ for a given sign vector $s \in \{\pm 1\}^d$. Namely, one has that
    \begin{align*}
        S &=\left\{ x \in \mathbb{R}^d \mid \max_{s\in\{\pm 1\}^d}\left\{\frac{1}{2}\|x\|_2^2 + \sum_{i=1}^d (w)_i(s)_i(x)_i\right\} \le C^2 \right\}\\
        &=\bigcap_{s \in \{\pm 1\}^d} \left\{ x \in \mathbb{R}^d \mid \frac{1}{2}\|x\|_2^2 + \sum_{i=1}^d (s)_i (w)_i (x)_i \le C^2 \right\}\\
        &=\bigcap_{s \in \{\pm 1\}^d} \left\{ x \in \mathbb{R}^d \mid \|x\|_2^2 + 2\sum_{i=1}^d (s)_i (w)_i (x)_i + \|w\|_2^2\le 2C^2 + \|w\|_2^2\right\}\\
        &=\bigcap_{s \in \{\pm 1\}^d} \left\{ x \in \mathbb{R}^d \mid \left\| x + s \circ w \right\|_2^2 \le 1 \right\}
    \end{align*}
    where the first equality uses the identity $|u|=\max_{s\in\{\pm 1\}} su$ component-wise, the second rewrites this as an intersection, the third multiplies by two and adds the constant $\|w\|_2^2$ to each inequality, and the fourth completes the square {\color{blue} having $\|w\|_2^2 = \|s \circ w\|_2^2$} and notes that the definition of $C$ and the fact that $\|w\|_2^2 = C^2 (d^2+d)$ ensure that $1 = 2C^2 + \|w\|_2^2$. Since each shifted unit ball is $1$-strongly convex{\color{blue}~\cite[Proposition 1]{vial1982strong}}, so is $S$.
\end{proof}

\begin{lemma}[Exact Diameter of $S$]\label{lem:diam-SC}
The diameter of $S$ is $\mathrm{diam}(S) = 2C(2-\sqrt{2})$.
\end{lemma}
\begin{proof}
Since $S$ is centrally symmetric, its diameter is $\mathrm{diam}(S) = 2 \max_{x \in S} \|x\|_2$.
Hence we maximize $\|x\|_2$ subject to the defining constraint~\eqref{eq:S-simple}. Since the weights are increasing, one can lower bound the weighted sum $\sum_{i=1}^d (w)_i |(x)_i|$ in terms of this two-norm as
\[
\sum_{i=1}^d (w)_i |(x)_i| \ge (w)_1 \sum_{i=1}^d |(x)_i| = C\sqrt{2} \|x\|_1 \ge C\sqrt{2} \|x\|_2.
\]
Both inequalities above hold with equality if and only if $(x)_i=0$ for all $i\geq 2$.
Applying this bound to the defining constraint of~\eqref{eq:S-simple}, every $x\in S$ must have  $\frac{1}{2}\|x\|_2^2 + C\sqrt{2} \|x\|_2 \le C^2$. Completing the square gives $\frac{1}{2}(\|x\|_2 + C\sqrt{2})^2 - C^2 \le C^2$, which simplifies directly to $(\|x\|_2 + C\sqrt{2})^2 \le 4C^2$. Hence $\|x\|_2 \le C(2-\sqrt{2})$. This upper bound is attained when the support of $x$ is restricted to the first coordinate. The boundary points $(\pm C(2-\sqrt{2}), 0, \dots, 0) \in S$ then attain $\mathrm{diam}(S) = 2C(2-\sqrt{2})$.
\end{proof}

\subsection{A Zero-Chain Property}\label{subsec:zero-chain-SC}
Lemma~\ref{lem:LMO-SC} establishes that linear minimization over the constructed $S$ corresponds to a thresholding operation. This form is essential for making $S$ a hard instance for all LMO-based algorithms. In the subsequent Lemma~\ref{lem:chain-SC}, we present our key technical result: this LMO has a ``zero-chain'' property that prevents an \texttt{LMO-span} algorithm from discovering more than one new nonzero coordinate per iteration.

\begin{lemma}[Exact LMO]\label{lem:LMO-SC}
For any search direction $p \in \mathbb{R}^d \setminus \{0\}$, the linear minimization oracle
$z = \mathtt{LMO}_S(p) \in \argmin_{x\in S} \langle p, x\rangle$
is given coordinatewise by the thresholding operator
\begin{equation}\label{eq:LMO}
(z)_i = -\mathrm{sign}((p)_i) \max\left(0, \frac{|(p)_i|}{\lambda} - (w)_i \right),
\end{equation}
where $\lambda > 0$ is the unique KKT multiplier ensuring $\frac{1}{2}\|z\|_2^2 + \sum_{i=1}^d (w)_i |(z)_i| = C^2$.
\end{lemma}
\begin{proof}
Let $h(x) = \frac{1}{2}\|x\|_2^2 + \sum_{i=1}^d (w)_i |(x)_i|$. The LMO solves $\min_{x} \langle p, x \rangle$ subject to $h(x) \le C^2$. Since the feasible region is compact, this is attained by some $z$. Since the constraint set is strongly convex and the objective is a non-constant linear function (i.e., $p\neq 0$), this $z$ is the unique minimizer. Further, it must lie on the boundary of the set, making the constraint $h(z)=C^2$ active. Slater's condition holds here as the origin is in the interior ({\color{blue} $h$ is continuous with} $h(0) = 0 < C^2$). Hence, the KKT conditions{\color{blue}~\cite[Corollary 28.2.1]{rockafellar1970convex}} provide a multiplier $\lambda\geq 0$ with $0\in p + \lambda \partial h(z)$. Since $p \neq 0$, $\lambda$ must be positive. The Lagrangian is
\[
\mathcal{L}(x, \lambda) = \langle p, x \rangle + \lambda \left( h(x) - C^2 \right).
\]
We consider the optimality condition for each $(z)_i$ with coordinate $i\in\{1,\dots, d\}$ separately. These require that $0 \in (p)_i + \lambda ((z)_i + (w)_i \partial |(z)_i|)$, giving the key inclusion $-\frac{(p)_i}{\lambda} \in (z)_i + (w)_i \partial |(z)_i|$.
We consider the three possible cases on the value of $(z)_i$: positive, negative, or equal to zero.
\begin{itemize}
    \item \textbf{Case $(z)_i > 0$:} Here $\partial |(z)_i| = \{1\}$, requiring $-\frac{(p)_i}{\lambda} = (z)_i + (w)_i$. Since $(z)_i > 0$ and $(w)_i > 0$, this implies $(p)_i = -\lambda ((w)_i+(z)_i) <0$. Thus $\mathrm{sign}((p)_i) = -1$ and so $(z)_i = \frac{|(p)_i|}{\lambda} - (w)_i$.
    \item \textbf{Case $(z)_i < 0$:} Here $\partial |(z)_i| = \{-1\}$, requiring $-\frac{(p)_i}{\lambda} = (z)_i - (w)_i$. Since $(z)_i < 0$ and $(w)_i > 0$, this implies $(p)_i = \lambda ((w)_i-(z)_i) >0$. Thus $\mathrm{sign}((p)_i) = 1$ and so $(z)_i = -\left(\frac{|(p)_i|}{\lambda} - (w)_i\right)$.
    \item \textbf{Case $(z)_i = 0$:} Here $\partial |(z)_i| = [-1, 1]$, requiring $-\frac{(p)_i}{\lambda} \in [-(w)_i, (w)_i]$. This directly evaluates to the condition $\frac{|(p)_i|}{\lambda} \le (w)_i$.
\end{itemize}
Combining these three cases gives the coordinatewise thresholding form in~\eqref{eq:LMO}.

{\color{blue} Finally, we establish uniqueness of the multiplier $\lambda$. The key observation is that the constraint must be active at the point $z(\lambda)$ above, by complementary slackness.} Since {\color{blue}$H(u) = \frac{1}{2}\|z(u)\|_2^2 + \sum_{i=1}^d (w)_i |(z(u))_i|$} is continuous, {\color{blue} strictly} decreasing in {\color{blue}$u$} whenever {\color{blue}$H(u)$} is positive, and ranges from $\infty$ to $0$, there is a unique {\color{blue}$u > 0$} with the defining constraint active {\color{blue}$H(u) = C^2$}. {\color{blue} Hence $\lambda$ must be exactly this value $u$.}
\end{proof}

\begin{lemma}[Zero-Chain Property]\label{lem:chain-SC}
Suppose at iteration $t \leq  d-2$ {\color{blue}of some \texttt{LMO-span} method, all past iterations $k \le t$ have $x_k$ and $z_k$ with identical values $(\cdot)_i$ across each of their coordinate indices $i \ge t+1$.} Then for any search direction $p_t$ satisfying~\eqref{eq:p-def}, the LMO point $z_{t+1} = \mathtt{LMO}_S(p_t)\in \argmin_{z \in S} \langle p_t, z \rangle$ satisfies
\[
(z_{t+1})_i = 0 \qquad \forall i\in\{ t+2,\dots,d\}.
\]
\end{lemma}
\begin{proof}
Noting $x_0=0$, any valid search direction $p_t$ must take the form $p_t = \sum_{k=0}^t \alpha_k \nabla f(x_k) + \sum_{k=1}^t \beta_k x_k + \sum_{k=1}^t \gamma_k z_k$ for some multipliers $\alpha_k,\beta_k,\gamma_k\in\mathbb{R}$. 
Letting $b_k,c_k\in\mathbb{R}$ be the common values of $(x_k)_i = b_k$ and $(z_k)_i=c_k$ for all $i \ge t+1$ and $x_\star = \nu \mathbf{1}$, each past gradient $\nabla f(x_k) = x_k -x_\star$ has $(\nabla f(x_k))_i = b_k-\nu$ for all $i \ge t+1$. Then $p_t$ must have a common value among its latter coordinates, which we denote by
\[
(p_t)_i = \sum_{k=0}^t \alpha_k (b_k-\nu) + \sum_{k=1}^t \beta_k b_k + \sum_{k=1}^t \gamma_k c_k =: \sigma.
\]
If $\sigma = 0$, the LMO~\eqref{eq:LMO} evaluates to $0$ for all remaining coordinates, satisfying the lemma.

Now assume $\sigma \neq 0$, and suppose for contradiction that $(z_{t+1})_j \neq 0$ for some $j \ge t+2$. 
Since $|(p_t)_j| = |\sigma|$, the thresholding operator formula~\eqref{eq:LMO} requires this coordinate to satisfy $\frac{|\sigma|}{\lambda} > (w)_j$. 
Since the weights $(w)_i$ are increasing, $\frac{|\sigma|}{\lambda} > (w)_{t+2} > (w)_{t+1}$. It follows that for coordinate $t+1$, which also has $|(p_t)_{t+1}| = |\sigma|$, the LMO formula~\eqref{eq:LMO} must return a strictly positive absolute value
\[
|(z_{t+1})_{t+1}| = \frac{|\sigma|}{\lambda} - (w)_{t+1} > (w)_{t+2} - (w)_{t+1} > 0.
\]
This strict inequality and the monotonicity of $u\mapsto \frac{1}{2}u^2+(w)_{t+1}u$ on $u>0$ imply that
\begin{align*}
&\frac{1}{2}|(z_{t+1})_{t+1}|^2 + (w)_{t+1} |(z_{t+1})_{t+1}|\\
&> \frac{1}{2}((w)_{t+2} - (w)_{t+1})^2 + (w)_{t+1}((w)_{t+2} - (w)_{t+1}) \\
&= \frac{1}{2} ((w)_{t+2}^2 - 2(w)_{t+2}(w)_{t+1} + (w)_{t+1}^2) + (w)_{t+1}(w)_{t+2} - (w)_{t+1}^2 \\
&= \frac{1}{2} ((w)_{t+2}^2 - (w)_{t+1}^2) = C^2
\end{align*}
where the final equality uses that by definition, adjacent squared weights satisfy $(w)_{t+2}^2-(w)_{t+1}^2=2C^2$. However, this strict inequality contradicts the feasibility of the LMO output $z_{t+1}$ as the sum $\sum_{i=1}^d \left( \frac{1}{2}(z)_i^2 + (w)_i |(z)_i| \right)$ must also strictly exceed $C^2$ since all other summands are nonnegative. Hence, the LMO output must have $(z_{t+1})_j = 0$ for all $j \ge t+2$.
\end{proof}

\subsection{Proof of Theorem~\ref{thm:SC}}\label{subsec:proof-SC}
We first inductively show that for any considered \texttt{LMO-span} method, at each iteration $t$, $(x_t)_i = 0$ for all coordinates $i \ge t+1$ and $(z_{t+1})_i = 0$ for all coordinates $i\geq t+2$.
When $t=0$, $x_0 = 0$ and so $(x_0)_i=0$ for $i\ge 1$. Applying Lemma~\ref{lem:chain-SC} then gives the needed property for $z_1$. Now, suppose {\color{blue} for the sake of induction that at some iteration $t>0$,} all $k\leq t-1$ have $x_k$ supported only on its first $k$ coordinates and $z_{k+1}$ supported only on its first $k+1$ coordinates. Since $x_{t} \in \operatorname{conv}\{0, z_1, \dots, z_{t}\}$ and each of these LMO solutions is zero on all coordinates greater than $t$, it follows that $(x_{t})_i = 0$ for $i \ge t+1$. Then applying Lemma~\ref{lem:chain-SC} restricts the support of $z_{t+1}$, guaranteeing that this returned LMO point has $(z_{t+1})_i = 0$ for all $i \ge t+2$, completing the induction.

From this, we can conclude that for any iteration $t \le d-1$, the suboptimality is lower bounded because the remaining $d-t$ coordinates of $x_t$ are all forced to equal zero:
\begin{equation}\label{eq:tailobj}
f(x_t) - f(x_\star) = \frac{1}{2} \|x_t - x_\star\|_2^2 \ge \frac{1}{2} \sum_{i=t+1}^d (0 - \nu)^2 = \frac{d-t}{2} \nu^2.
\end{equation}
Therefore, to establish the theorem, it suffices to provide a lower bound on $\nu^2$.

Let $W = \sum_{i=1}^d (w)_i$. Applying the quadratic formula to~\eqref{eq:nu-def} and rationalizing the numerator provides the following expression for $\nu$ of
\begin{equation}\label{eq:nu-quad-form}
\nu = \frac{-W + \sqrt{W^2 + 2 d C^2}}{d} = \frac{2 C^2}{W + \sqrt{W^2 + 2 d C^2}}.
\end{equation}
The summation $W$ is upper bounded by $\frac{2\sqrt{2}}{3}C(d+1)^{3/2}$ via the following integral upper bound
\[
\sum_{i=1}^d \sqrt{i} \le \int_0^d \sqrt{x+1} \, dx = \left[ \frac{2}{3}(x+1)^{3/2} \right]_0^d = \frac{2}{3}\left( (d+1)^{3/2} - 1 \right) < \frac{2}{3}(d+1)^{3/2}.
\]
Since the formula~\eqref{eq:nu-quad-form} for $\nu$ is decreasing in $W$, we can lower bound $\nu^2$ by
$$ \nu^2 \geq \frac{(2C^2)^2}{C^2\left(\frac{2\sqrt{2}}{3}(d+1)^{3/2} + \sqrt{\frac{8}{9}(d+1)^3 + 2d}\right)^2} \geq \frac{9 C^2}{8(d+2)^3}$$
where the final step bounds $\left(\frac{2\sqrt{2}}{3}(d+1)^{3/2} + \sqrt{\frac{8}{9}(d+1)^3 + 2d}\right)^2 \leq \frac{32}{9}(d+2)^{3}$, which is valid\footnote{To verify this algebraic bound, let $x=\frac{8}{9}(d+1)^3$. Then the left-hand side equals $(\sqrt{x} + \sqrt{x+2d})^2=2x+2d+2\sqrt{x^2+2dx}$. Since $\sqrt{x^2+2dx} \leq x+d$, this is bounded above by $\frac{32}{9}(d+1)^3 + 4d$. For $d\geq 1$, this can be further bounded above by $\frac{32}{9}(d+2)^3$, giving the final simplified bound.} for all $d\geq 1$. Finally, we recall from Lemma~\ref{lem:diam-SC} that the diameter of $S$ is $2C(2-\sqrt{2})$. Rearranging, we have that $C^2 = \frac{3+2\sqrt{2}}{8} \mathrm{diam}(S)^2$. Hence by~\eqref{eq:tailobj}, the suboptimality after $t$ iterations of any \texttt{LMO-span} is bounded by
\[
  f(x_t)-f(x_\star) \geq \frac{d-t}{2}\ \frac{9(3+2\sqrt{2})\mathrm{diam}(S)^2}{64(d+2)^3} \geq \frac{2}{5}\frac{(d-t)\mathrm{diam}(S)^2}{(d+2)^3}
\]
with the last inequality reducing the absolute constant to a simpler lower bound, $\frac{9(3+2\sqrt{2})}{128}\geq \frac{2}{5}$.

For a fixed budget of iterations $T \ge 1$, fix $d=2(T+1)$ in the above construction. Then our suboptimality lower bound at $t=T$ becomes
\[
    f(x_T) - f(x_\star) \geq \frac{2}{5}\frac{(T+2)\ \mathrm{diam}(S)^2}{(2T+4)^3} = \frac{1}{20}\frac{\mathrm{diam}(S)^2}{(T+2)^2}.
\]

\section{A Resisting Oracle Extension to Deterministic FO-LMO Methods}\label{sec:general}

The lower bound provided in Theorem~\ref{thm:SC} establishes that no \texttt{LMO-span} method can guarantee suboptimality less than $\Omega(L\, \mathrm{diam}(S)^2/T^2)$. However, if the span restriction is relaxed, one can design a method capable of exactly solving any single fixed hard problem. Trivially, one could consider the algorithm that always returns $\nu\mathbf{1}$ for every problem instance. Although often not an effective algorithm, this would exactly solve our previously proposed hard instance. {\color{blue} So we require a family of hard instances.} 

To resolve this, Section~\ref{subsec:permuted-problem} {\color{blue}provides an adversarial process for constructing variants} of our previous hard problem. Section~\ref{subsec:general-zero-chain} derives a ``zero-chain'' property for these adversarial LMOs, resisting any \texttt{FO-LMO} method. Finally, Section~\ref{subsec:general-thm} {\color{blue} shows that for any \texttt{FO-LMO}, our resisting construction provides a tailored hard instance,} proving Theorem~\ref{thm:main}. {\color{blue} Again, it suffices to consider $1$-smooth, $1$-strongly convex functions and $1$-strongly convex sets by Lemma~\ref{lem:rescaling}.}

\subsection{A Permuted Family of Hard Problem Instances} \label{subsec:permuted-problem}
Let $\mathcal{P}_d$ denote the symmetric group of all permutations on $\{1, \dots, d\}$. Given a permutation $\pi \in \mathcal{P}_d$, we define the permuted constraint set
\begin{equation} \label{eq:S-perm}
S_\pi := \left\{ x \in \mathbb{R}^d \mid \frac{1}{2}\|x\|_2^2 + \sum_{i=1}^d (w)_{\pi(i)} |(x)_i| \le C^2 \right\},
\end{equation}
where $C$ and $(w)_i = C\sqrt{2i}$ are defined identically to~\eqref{eq:params}. Since $S_\pi$ is a coordinate permutation of the original set $S$, it retains the exact same diameter and $\alpha=1$ strong convexity as $S$. {\color{blue} Likewise, Lemma~\ref{lem:LMO-SC} under this permutation ensures, for any search direction $p \in \mathbb{R}^d \setminus \{0\}$, the linear minimization oracle solution $z$ is given by
\begin{equation} \label{eq:LMO-permuted-solution}
(z)_i = -\mathrm{sign}((p)_i) \max\left(0, \frac{|(p)_i|}{\lambda} - (w)_{\pi(i)} \right),
\end{equation}
where $\lambda > 0$ is the unique KKT multiplier ensuring $\frac{1}{2}\|z\|_2^2 + \sum_{i=1}^d (w)_{\pi(i)} |(z)_i| = C^2$.}

For any $d\geq 3$,\footnote{{\color{blue}Note that restricting our attention to settings with $d\geq 3$ is sufficient for our aim of proving Theorem~\ref{thm:main}. There we assume $T\geq 1$, so $d=2(T+1)\geq 4$.}} we define the $1$-smooth, $1$-strongly convex objective function as 
\begin{equation} \label{eq:f-perm}
f(x) := \frac{1}{2}\|x - M \mathbf{1}\|_2^2,
\end{equation}
where $M := \frac{\rho}{1-\rho} (w)_d$ with $\rho \in (0, 1)$ as the unique scalar satisfying
\begin{equation} \label{eq:rho-def}
\sum_{i=1}^d \left[ \frac{1}{2}\rho^2 ((w)_d - (w)_i)^2 + \rho (w)_i ((w)_d - (w)_i) \right] = C^2.
\end{equation}
Existence and uniqueness follow since the above polynomial is increasing for $\rho>0$, evaluating to $0$ at $\rho=0$, and $\sum_{i=1}^d \frac{1}{2}((w)_d^2 - (w)_i^2) = C^2 \frac{d(d-1)}{2} > C^2$ at $\rho=1$ since $d \ge 3$. 
The following lemma computes the true constrained optimum $x_\star^{(\pi)} = \argmin_{x \in S_\pi} f(x)$ and provides a lower bound on $\rho$.

\begin{lemma}[The Permuted Problem Minimizer] \label{lem:perm-optimum}
For any $\pi \in \mathcal{P}_d$, the optimal solution $x_\star^{(\pi)}$ to $\min_{x \in S_\pi} f(x)$ is unique and given coordinatewise by
\[
(x_\star^{(\pi)})_i = \rho \left( (w)_d - (w)_{\pi(i)} \right) \quad \forall i \in \{1, \dots, d\}.
\]
Furthermore, the scalar $\rho\in (0,1)$ is bounded below by $\rho \ge 2/d^2$.
\end{lemma}
\begin{proof}
{\color{blue} Note that for any permutation, this minimization problem is convex and possesses a Slater point (the origin). So it suffices to } verify optimality of $x_\star^{(\pi)}$ by showing {\color{blue} that the first-order optimality condition for minimizing $f(x)$ over $h(x) = \frac{1}{2}\|x\|_2^2 + \sum_{i=1}^d (w)_{\pi(i)} |(x)_i| \le C^2$ holds. So we will show existence of} some subgradient $\zeta \in \partial ( \sum_{i=1}^d (w)_{\pi(i)} |(\cdot)_i| )(x_\star^{(\pi)})$ and {\color{blue} Lagrange} multiplier $\gamma\geq 0$ such that
$$ x_\star^{(\pi)} - M \mathbf{1} + \gamma (x_\star^{(\pi)} + \zeta) = 0. $$
In particular, we set $\gamma = \frac{\rho}{1-\rho}$ {\color{blue} and verify this condition coordinatewise}.
Since $(w)_d \ge (w)_{\pi(i)}$, note that $(x_\star^{(\pi)})_i \ge 0$. {\color{blue}First consider any $i$ with} $(w)_{\pi(i)} < (w)_d$. In this case, $(x_\star^{(\pi)})_i > 0$ and we can set the required subgradient as $(\zeta)_i = (w)_{\pi(i)}$. Then the {\color{blue}$i$th coordinate's} KKT condition holds since
{\color{blue}
\begin{align*}
    &\left(x_\star^{(\pi)} - M \mathbf{1} + \gamma (x_\star^{(\pi)} + \zeta)\right)_i\\
    &=\rho((w)_d - (w)_{\pi(i)}) - M + \frac{\rho}{1-\rho}(\rho((w)_d - (w)_{\pi(i)}) + (w)_{\pi(i)})\\
    &= \left(\left(\rho + \frac{\rho^2}{1-\rho}\right)(w)_d - M\right)-\left(\rho + \frac{\rho^2}{1-\rho} - \frac{\rho}{1-\rho}\right)(w)_{\pi(i)}=0
\end{align*}
where the first equality applies our choice of $x_\star^{(\pi)}$ and $\zeta$, the second rearranges terms, and the third substitutes $\rho + \frac{\rho^2}{1-\rho} = \gamma = \frac{\rho}{1-\rho}$ and $M=\gamma (w)_d$.}

For the unique component $i$ where $(w)_{\pi(i)} = (w)_d$, we have $(x_\star^{(\pi)})_i = 0$. The {\color{blue} associated $i$th} KKT condition requires $-M + \gamma (\zeta)_i = 0$, implying $(\zeta)_i = M/\gamma = (w)_d$. Since $(w)_d \in [-(w)_d, (w)_d]$, this is a valid subgradient. Finally, substituting $(x_\star^{(\pi)})_i$ into the defining boundary constraint of $S_\pi$ yields the polynomial equation~\eqref{eq:rho-def}, confirming feasibility.

Finally, we lower bound $\rho$. Observe that for each coordinate $i$, one has
\begin{align*}
&\rho ((w)_d - (w)_i) \left[ \frac{1}{2}\rho ((w)_d - (w)_i) + (w)_i \right]\\
&\le \rho ((w)_d - (w)_i) \left[ \frac{1}{2} ((w)_d - (w)_i) + (w)_i \right] = \frac{1}{2}\rho ((w)_d^2 - (w)_i^2).
\end{align*}
Summing over $i$ yields $C^2 \le \frac{1}{2}\rho \sum_{i=1}^d ((w)_d^2 - (w)_i^2) = \frac{1}{2}\rho C^2 d(d-1)$, and so $\rho \geq \frac{2}{d(d-1)} \ge \frac{2}{d^2}$.
\end{proof}

\subsection{A Generalized Zero-Chain Property} \label{subsec:general-zero-chain}
Any first-order query at an iterate $x_k$ yields a gradient $\nabla f(x_k) = x_k - M \mathbf{1}$. Since $M$ is independent of the choice of permutation $\pi$, this gradient for any $x_k$ reveals no information about $\pi$. Thus \texttt{FO-LMO} algorithms must use the LMO to reveal information about $x^{(\pi)}_\star$.
We formalize the information revealed by the LMO via an adversarial ``resisting oracle'' that selects the permutation $\pi$ at runtime. The oracle maintains a set $K_t \subseteq \{1, \dots, d\}$ of assigned coordinates, initially empty.

At each iteration, when the algorithm queries the LMO with an arbitrary vector $p_t \neq 0$, the oracle selects an unassigned coordinate attaining the largest query magnitude
\[
i_\star \in \operatorname{argmax}_{i \notin K_t} |(p_t)_i|.
\]
The adversary assigns this coordinate the smallest available weight, setting $\pi(i_\star) = |K_t| + 1$ and $K_{t+1}=K_t\cup\{i_\star\}$.
{\color{blue} Lemma~\ref{lem:resisting-oracle} below shows that, regardless of how the remainder of $\pi$ is assigned, the LMO point~\eqref{eq:LMO-permuted-solution} remains the same---assigning $(z_{t+1})_{i}=0$ for every unassigned index.}
Alternatively, if given $p_t=0$, we define the oracle output as $z_{t+1}=0$, which is feasible for all $S_\pi$ and reveals no new information.

\begin{lemma}[Generalized Zero-Chain Property]\label{lem:resisting-oracle}
Under the resisting oracle's assignment, for any $p_t$, the returned LMO point $z_{t+1} = \mathtt{LMO}_{S_\pi}(p_t)$ has $(z_{t+1})_j = 0$ for all unassigned coordinates $j \notin K_{t+1}$, regardless of how the remainder of $\pi$ is eventually assigned.
\end{lemma}
\begin{proof}
This result is immediate if $p_t=0$ as $z_{t+1}=0$.
As a result, we can assume $p_t\neq 0$ and let $\lambda > 0$ be the KKT multiplier for the LMO query. Assume there exists some $j \notin K_{t+1}$ with $(z_{t+1})_j \neq 0$. From the coordinatewise thresholding formula derived in~\eqref{eq:LMO}, this requires $\frac{|(p_t)_j|}{\lambda} > (w)_{\pi(j)}$.
Since $i_\star$ was chosen as a maximizer of the search direction's magnitude among all unassigned coordinates, we have $|(p_t)_{i_\star}| \ge |(p_t)_j|$. Furthermore, the oracle assigns weights monotonically. So any unassigned coordinate $j$ will eventually receive a weight $(w)_{\pi(j)} \ge (w)_{|K_t|+2}$. Thus, we have
\[
\frac{|(p_t)_{i_\star}|}{\lambda} \ge \frac{|(p_t)_j|}{\lambda} > (w)_{\pi(j)} \ge (w)_{|K_t|+2}.
\]
It follows that the LMO output $z_{t+1}$ at coordinate $i_\star$ is nonzero and lower bounded absolutely by
\[
|(z_{t+1})_{i_\star}| = \frac{|(p_t)_{i_\star}|}{\lambda} - (w)_{|K_t|+1} > (w)_{|K_t|+2} - (w)_{|K_t|+1} > 0.
\]
Since the function $u \mapsto \frac{1}{2}u^2 + (w)_{|K_t|+1}u$ is increasing for $u > 0$, we have that
\[
\frac{1}{2}(z_{t+1})_{i_\star}^2 + (w)_{|K_t|+1} |(z_{t+1})_{i_\star}| > \frac{1}{2}((w)_{|K_t|+2}^2 - (w)_{|K_t|+1}^2) = C^2,
\]
where the last equality follows from the definition $(w)_k^2 = 2C^2 k$.
However, this violates the constraint $\sum_{i=1}^d (\frac{1}{2}(z)_i^2 + (w)_{\pi(i)}|(z)_i|) \le C^2$ since all other terms in the sum are nonnegative. From this contradiction, we conclude that $z_{t+1}$ is zero on all remaining coordinates, regardless of how the permutation $\pi$ is completed.
\end{proof}

\subsection{Proof of Theorem~\ref{thm:main}} \label{subsec:general-thm}
By Lemma~\ref{lem:resisting-oracle}, our resisting LMO reveals at most one coordinate of the permutation $\pi$ at each iteration. After $T$ iterations, the set of assigned coordinates $K_T$ has size at most $T$. Let $U = \{1, \dots, d\} \setminus K_T$ denote the remaining coordinates, with $m = |U| \ge d - T$.

Any considered \texttt{FO-LMO} method must set $x_T$ deterministically as a function of the history of oracle responses. Crucially, while the algorithm must fix $(x_T)_i$ for $i \in U$, the resisting oracle remains free to select any completion of the permutation $\pi$ over $U$. Let $\mathcal{P}_U\subseteq \mathcal{P}_d$ denote all such permutations. If some completion $\pi\in\mathcal{P}_U$ has $x_T\not\in S_\pi$, then the adversarial oracle can select that $\pi$, making $x_T$ infeasible and the theorem hold trivially. Hence, we can assume for all $\pi\in \mathcal{P}_U$ that $x_T\in S_\pi$. Note that the optimality of $x^{(\pi)}_\star$ guarantees $\langle \nabla f(x^{(\pi)}_\star), x_T - x^{(\pi)}_\star\rangle\geq 0$. Since $x_T\in S_\pi$, the suboptimality for any $\pi\in\mathcal{P}_U$ is then bounded by
{\color{blue}
\begin{align}
f(x_T) - f(x_\star^{(\pi)})
&= \left\langle \nabla f(x_\star^{(\pi)}), x_T - x_\star^{(\pi)} \right\rangle
+ \frac{1}{2}\|x_T - x_\star^{(\pi)}\|_2^2 \nonumber \\
&\geq \frac{1}{2}\|x_T - x_\star^{(\pi)}\|_2^2 \nonumber \\
&\geq \frac{1}{2} \sum_{j\in U} \left( (x_T)_{j} - (x^{(\pi)}_\star)_{j} \right)^2 . \label{eq:sc-lower-bound}
\end{align}
}
Let $\pi_\star \in \mathcal{P}_U$ denote the permutation choice maximizing this lower bound. {\color{blue} Recall that the variance of $m$ scalars $X_i$ is defined as their average deviation squared from their mean. We denote this as follows, noting a convenient reformulation as well,
$$\mathrm{Var}(X) = \frac{1}{m}\sum_{i}\left(X_i- \frac{1}{m}\sum_j X_j\right)^2 = \frac{1}{2m^2}\sum_{i,j}(X_i-X_j)^2. $$
}
Then, we have
\begin{align*}
    f(x_T) - f(x_\star^{(\pi_\star)}) &\geq \frac{1}{2} \sum_{j\in U} \left( (x_T)_{j} - (x^{(\pi_\star)}_\star)_{j} \right)^2\\
    &\geq \frac{1}{|\mathcal{P}_U|}\sum_{\pi\in\mathcal{P}_U} \left(\frac{1}{2}\sum_{j\in U} \left( (x_T)_{j} - (x^{(\pi)}_\star)_{j} \right)^2\right)\\
    & = \frac{1}{2}\sum_{j\in U} \frac{1}{|\mathcal{P}_U|}\sum_{\pi\in\mathcal{P}_U}  \left( (x_T)_{j} - (x^{(\pi)}_\star)_{j} \right)^2\\
    &\geq \frac{m}{2}\mathrm{Var}\left\{ \rho((w)_d - (w)_i) \right\}_{i = d-m+1}^d \\
    &=\frac{\rho^2 m}{2} \mathrm{Var}\{(w)_i\}_{i = d-m+1}^d\\
    &= \rho^2 m C^2 \mathrm{Var}\{\sqrt{i}\}_{i = d-m+1}^d
\end{align*}
where the first inequality is by~\eqref{eq:sc-lower-bound}, the second inequality lower bounds this maximal $\pi_\star$ by the average lower bound over $\mathcal{P}_U$, and the third inequality notes that each inner sum is minimized when $(x_T)_{j}$ is the average of $\{\rho((w)_d - (w)_{d-m+1}),\dots, \rho((w)_d - (w)_{d})\}$ yielding the stated variance lower bound.

To bound the above variance, we use that $\mathrm{Var}(X) = \frac{1}{2m^2}\sum_{i,j}(X_i-X_j)^2$. For $X_i=\sqrt{i}$, we can lower bound $(\sqrt{i}-\sqrt{j})^2 =(i-j)^2/(\sqrt{i}+\sqrt{j})^2\geq (i-j)^2/4d$. So $\mathrm{Var}(X)\geq \mathrm{Var}(i)/4d =  (m^2-1)/48d$, using that the variance of $m$ contiguous integers\footnote{{\color{blue} This variance of $m$ consecutive integers can be computed directly as
$$ \mathrm{Var}(i) = \frac{1}{2m^2}\sum_{i,j=1}^m(i-j)^2 = \frac{1}{2m^2}\sum_{\delta=0}^{m-1} 2(m-\delta) \delta^2 = \frac{1}{m}\sum_{\delta=0}^{m-1}\delta^2 - \frac{1}{m^2} \sum_{\delta=0}^{m-1} \delta^3 = \frac{m^2-1}{12}$$
where the first equality is by definition, the second notes $\delta=|i-j|$ occurs for $2(m-\delta)$ choices of $i,j$, and the final equalities rearrange terms and apply standard summation identities.}} is $(m^2-1)/12$.
Further, note that for $d=2T+2$, we have $m\geq T+2$. Then applying that $\rho\geq 2/d^2$ and $C^2 = \frac{3+2\sqrt{2}}{8} \mathrm{diam}(S_\pi)^2$, we conclude that there exists a problem instance resisting any given \texttt{FO-LMO} method, having
\begin{align*}
    f(x_T) - f(x_\star^{(\pi_\star)}) &\geq \rho^2 m C^2 \frac{m^2-1}{48d} \geq \frac{3+2\sqrt{2}}{8} \frac{ \mathrm{diam}(S_\pi)^2}{384(T+1)^2} > \frac{1}{528} \frac{\mathrm{diam}(S_\pi)^2}{(T+1)^2}
\end{align*}
where the last inequality above just reduces to a simpler fractional coefficient. Hence, $f$ and $S_{\pi_\star}$ together constitute a hard problem instance for the given first-order method, adversarially constructed. The case of general $L,\alpha>0$ follows by considering the rescaling formulas~\eqref{eq:rescalings}.

\section{An Extension of Lower Bounds to Modestly Smooth Sets}\label{sec:smooth}
Finally, we provide a direct approach to extend our span-based lower bounds to guarantees against methods applied to $\beta$-smooth sets. We do this by considering the smoothing of hard instances given by taking Minkowski sums with a ball $B(0,1/\beta)$. For sufficiently large values of $\beta=\Omega(1/\sqrt{\varepsilon})$, we show that this perturbation of the problem instance cannot notably improve worst-case performance. As a result, we find that no acceleration is possible for \texttt{LMO-span} methods on only modestly smooth sets. We leave open whether acceleration is possible when sets possess constant levels of smoothness. The numerical survey of~\cite{luner2024performance} using performance estimation techniques suggested that for small values of $T$, no clear big-O acceleration could be numerically identified for more general ranges of $\beta$.

Below, we provide lower bounds for optimization over smooth convex sets and over smooth, strongly convex sets. In the $\beta=\Omega(1/\sqrt{\varepsilon})$ regime, our smooth convex set lower bound matches the general optimal complexity for optimization over convex sets of~\cite{jaggi2013revisiting,lan2013complexity}. Similarly, in this regime, our smooth, strongly convex set lower bound matches the strongly convex complexity bound established in the previous section of $\Omega(\sqrt{L\, \mathrm{diam}(S)^2/\varepsilon})$. 

\begin{theorem}\label{thm:smooth}
For any $d > 1$ and $\beta > 0$, there exist a convex $\beta$-smooth set $S_\beta$ and a $1$-smooth, $1$-strongly convex function $f_\beta$ such that every \texttt{LMO-span} method applied to $(f_\beta, S_\beta)$ starting from $x_0=0$ has for all $t\leq d-1$
    \begin{equation}
        f_\beta(x_t) - \min_{x\in S_\beta}f_\beta(x) \geq  \max \left(0, \sqrt{\frac{d-t}{2}}\left(\frac{\mathrm{diam}(S_\beta) -  2/\beta}{\sqrt{2}d}+\frac{1}{\beta\sqrt{d}}\right)-\frac{1}{\sqrt{2}\beta}\right)^2  . 
    \end{equation}
    In particular, for any fixed budget $T \ge 1$, there exist $S_\beta$ and $f_\beta$ in dimension $d=2T$ such that
    \[
        f_\beta(x_T) - \min_{x\in S_\beta}f_\beta(x) \geq \max\left(0,\frac{\mathrm{diam}(S_\beta) - 2/\beta}{4\sqrt{T}} - \frac{\sqrt{2}-1}{2\beta }\right)^2 .
    \]
\end{theorem}
\begin{theorem}\label{thm:mixed}
    For any $d > 1$ and $\beta > 0$, there exist a $\beta$-smooth, $1/(1+1/\beta)$-strongly convex set $S_\beta$ and a $1$-smooth, $1$-strongly convex function $f_\beta$ such that every \texttt{LMO-span} method applied to $(f_\beta, S_\beta)$ starting from $x_0 = 0$ has for all $t\leq d-1$
    \begin{equation}
        f_\beta(x_t) - \min_{x\in S_\beta}f_\beta(x) \geq \max\left(0,\sqrt{\frac{2}{5}\frac{(d-t)(\mathrm{diam}(S_\beta) -  2/\beta)^2}{(d+2)^3}} - \frac{1}{\sqrt{2}\beta}\right)^2  . 
    \end{equation}
    In particular, for any fixed budget $T \ge 1$, there exist $S_\beta$ and $f_\beta$ in dimension $d=2(T+1)$ such that
    \[
        f_\beta(x_T) - \min_{x\in S_\beta}f_\beta(x) \geq \max\left(0,\frac{\mathrm{diam}(S_\beta) - 2/\beta}{\sqrt{20}(T+2)} - \frac{1}{\sqrt{2}\beta }\right)^2  .
    \]
\end{theorem}
\noindent In both of these theorems, the lower bounds are only meaningful in the regime of $\beta=\Omega(1/\sqrt{\varepsilon})$ as otherwise, the two-term maximums above will take value zero, making the lower bound vacuous.

To prove these results, we use the following lemma which allows us to relate a set $S$ with a suitable zero-chain property to an approximate zero-chain property holding on the set $S_\beta = S+B(0,1/\beta)$ given by a Minkowski sum with a ball. Note that $S_\beta$ is $\beta$-smooth for any closed convex set $S$ by~\cite[Lemma 9]{liu2023gauges}. Further, the LMO for $S_\beta$ is given by summing the LMOs for $S$ and $B(0,1/\beta)$: by the additivity of linear minimization over Minkowski sums, it follows that
$$ \mathtt{LMO}_{S_\beta}(p) = \mathtt{LMO}_S(p) - \frac{1}{\beta} \frac{p}{\|p\|_2} . $$
We find that $\mathtt{LMO}_{S_\beta}(p)$ can inherit an approximate zero-chain property from $\mathtt{LMO}_S(p)$ as follows.

\begin{lemma}[Approximate Zero-Chain Property]\label{lem:approx-chain}
    Consider a compact convex set $S{\color{blue}\subseteq\mathbb{R}^d}$ and $f_\beta(x) = \frac{1}{2}\|x - \nu\mathbf{1}\|^2_2$ for some $\nu>0$. Let $S_\beta = S +B(0,1/\beta)$. Suppose that the base-set LMO has the property that: {\color{blue} if for each $k \leq t\leq d-2$, $x_k$ and $z_k$ each have identical values $(\cdot)_i$ across each of their coordinate indices $i \ge t+1$, then, for any $p_t$ satisfying~\eqref{eq:p-def},} $\hat z_{t+1} = \mathtt{LMO}_{S}(p_t)$ satisfies $(\hat z_{t+1})_i = 0$ for all $i \ge t+2$. Then, in this case, the smoothed-set LMO point $z_{t+1} = \mathtt{LMO}_{S_\beta}(p_t)$ satisfies
    \[
    (z_{t+1})_i = \delta_{t+1}\qquad \forall i \in\{ t+2,\dots,d\}
    \]
    where the common tail value $\delta_{t+1}$ is bounded by $|\delta_{t+1}| \le \frac{1}{\beta \sqrt{d-t}}$.
\end{lemma}

\begin{proof}
    By the hypothesis, for each $k \le t$, the vectors $x_k$ and $z_k$ each have all coordinates $i \ge t+1$ equal (within each vector). Since $x_\star = \nu \mathbf{1}$, the gradients $\nabla f_\beta(x_k) = x_k - x_\star$ also have $(\nabla f_\beta(x_k))_i$ constant among all $i \ge t+1$. As a result, any valid search direction $p_t$ satisfying~\eqref{eq:p-def}$\,$ is constant among its coordinates $i \ge t+1$. Denote this shared scalar value by $\sigma$.

    Let $\hat z_{t+1}=\mathtt{LMO}_S(p_t)$. By assumption, $(\hat z_{t+1})_i=0$ for all $i\ge t+2$. Since $p_t\neq 0$ under~\eqref{eq:p-def}, linear minimization over Minkowski sums gives
    \[
    z_{t+1}=\mathtt{LMO}_{S_\beta}(p_t)=\mathtt{LMO}_S(p_t)-\frac{1}{\beta}\frac{p_t}{\|p_t\|_2}
    =\hat z_{t+1}-\frac{1}{\beta}\frac{p_t}{\|p_t\|_2}.
    \]
    Therefore, for every $i \ge t+2$, the coordinate $(z_{t+1})_i$ is determined solely by the ball term, having
    $(z_{t+1})_i = - \frac{\sigma}{\beta \|p_t\|_2} =: \delta_{t+1}$.
    Hence, all such tail coordinates are equal.
    If $\sigma = 0$, then $\delta_{t+1} = 0$, satisfying the lemma's claim. If $\sigma\neq 0$, then since $p_t$ contains at least $d-t$ copies of $\sigma$ (at indices $t+1, \dots, d$), its two-norm is at least $\|p_t\|_2 \ge \sqrt{d-t}\, |\sigma|$. Rearrangement gives the claimed bound as $|\delta_{t+1}| = \frac{|\sigma|}{\beta \|p_t\|_2}
    \le \frac{|\sigma|}{\beta \sqrt{d-t}\, |\sigma|}
    = \frac{1}{\beta \sqrt{d-t}}.$
\end{proof}

\subsection{Proof of Theorem~\ref{thm:smooth}}
Given any {\color{blue}integer $d > 1$}, consider the simplex $S=\left\{x\in \mathbb{R}^d \mid \sum_{i=1}^d (x)_{i}\le 1, (x)_{i}\geq 0  \right\}$. This choice follows directly from that of Lan~\cite{lan2013complexity} where a simplex was shown to possess a zero-chain property, {\color{blue} sufficient for applying Lemma~\ref{lem:approx-chain}: Observe that each extreme point has at most one nonzero coordinate, so an adversarial LMO can return either the origin or the attaining basis vector with minimal index. If the latter coordinates of $x_k$ and $z_k$ are all constant, then the LMO can be chosen to return $(t+1)$th standard basis vector instead of any later basis vector. Hence the coordinates $i\geq t+2$ will always equal zero.}

From this, Lan provided a hard instance for any LMO method applied to general convex constrained optimization.
Fixing any $\beta>0$, we define the $\beta$-smooth set
$$ S_\beta := S + B(0,1/\beta) . $$
The diameter of $S_\beta$ is $\mathrm{diam}(S_\beta) = \mathrm{diam}(S)+2/\beta = \sqrt{2} + 2/\beta$. To keep the target minimizer on the boundary of this set,  we define the optimal solution $x_\star=(1/d + 1/(\beta\sqrt{d}))\mathbf{1}$. We continue as in previous constructions by setting $f_{\beta}(x):=\frac{1}{2}\|x-x_\star\|_2^2$.

Inductively applying Lemma~\ref{lem:approx-chain}, the iterate $x_t \in \operatorname{conv}\{x_0, z_1, \dots, z_t\}$ must have final $d-t$ coordinates $i\geq t+1$ equal with value bounded by $|(x_t)_i| \leq \max_{1 \le k \le t} |\delta_k| \le \frac{1}{\beta \sqrt{d-t+1}}$. Then the suboptimality $f_\beta(x_t) - f_\beta(x_\star) = \frac{1}{2}\|x_t-x_\star\|^2_2$ is lower bounded by
\begin{align*}
 f_\beta(x_t) - f_\beta(x_\star) &\ge \frac{1}{2} \sum_{i=t+1}^d \left((x_t)_i - \left(\frac{1}{d}+\frac{1}{\beta \sqrt{d}}\right)\right)^2     \\
 &\ge \frac{d-t}{2}\max \left(0, \frac{1}{d}+\frac{1}{\beta\sqrt{d}}-\frac{1}{\beta\sqrt{d-t+1}}\right)^2 \\
 &{\color{blue}= \max \left(0, \sqrt{\frac{d-t}{2}}\left(\frac{\mathrm{diam}(S_\beta) -  2/\beta}{\sqrt{2}d}+\frac{1}{\beta\sqrt{d}}\right)-\frac{1}{\sqrt{2}\beta}\sqrt{\frac{d-t}{d-t+1}}\right)^2}\\
 &\ge \max \left(0, \sqrt{\frac{d-t}{2}}\left(\frac{\mathrm{diam}(S_\beta) -  2/\beta}{\sqrt{2}d}+\frac{1}{\beta\sqrt{d}}\right)-\frac{1}{\sqrt{2}\beta}\right)^2
\end{align*}
{\color{blue} where the equality step distributes $(d-t)/2$ into the square and uses our formula for the diameter of $S_\beta$.}
Specializing to $t=T$ and $d=2T$, this provides a lower bound of
$$ f_\beta(x_T) - f_\beta(x_\star) \geq \max\left(0,\frac{\mathrm{diam}(S_\beta) -  2/\beta}{4\sqrt{T}} - \frac{\sqrt{2}-1}{2\beta }\right)^2  . $$

\subsection{Proof of Theorem~\ref{thm:mixed}}
Given any {\color{blue}integer $d > 1$}, consider the set $S$ previously defined in~\eqref{eq:S-simple} and the parameter $\nu$ defined by~\eqref{eq:nu-def}. Fix any $\beta > 0$ and define
$$ S_\beta := S + B(0,1/\beta)  . $$
From Lemma~\ref{lem:diam-SC}, it is immediate that the diameter of $S_\beta$ is $\mathrm{diam}(S_\beta) = \mathrm{diam}(S)+2/\beta = 2C(2-\sqrt{2}) + 2/\beta$. Further, $S_\beta$ is $\beta$-smooth and $\alpha = 1/(1+1/\beta)$-strongly convex by the calculus rules for Minkowski sums~\cite[Lemma 9]{liu2023gauges}.

{\color{blue} To place $x_\star$ on the boundary of $S_\beta$, we require a different choice of $\nu_\beta$ than the previous~\eqref{eq:nu-def} defining $\nu$.} We define $x_\star = \nu_\beta \mathbf{1}$ with $\nu_\beta$ as the unique positive root of
$$ \min_{x \in S} \|\nu_\beta \mathbf{1} - x\|_2 = \frac{1}{\beta}  . $$
Since $0 \in \mathrm{int}(S)$, the distance from $\eta \mathbf{1}$ to the set $S$ is zero for $0\leq \eta \leq \nu$. {\color{blue} Further since the distance function is convex and positive for $\eta>\nu$,} it must be strictly increasing for $\eta > \nu$, mapping $[\nu, \infty)$ to $[0, \infty)$. Thus, this equation uniquely defines a scalar $\nu_\beta > \nu$. Given this, we continue as in previous constructions by setting $f_\beta(x) := \frac{1}{2}\|x-x_\star\|_2^2$.

Inductively applying the approximate zero-chain property of $S_\beta$, the iterate $x_t \in \operatorname{conv}\{x_0, z_1, \dots, z_t\}$ must have final $d-t$ coordinates $i\geq t+1$ equal with value bounded by $|(x_t)_i| \leq \max_{1 \le k \le t} |\delta_k| \le \frac{1}{\beta \sqrt{d-t+1}}$. Then the suboptimality is lower bounded by
$$ f_\beta(x_t) - f_\beta(x_\star) \ge \frac{1}{2} \sum_{i=t+1}^d ((x_t)_i - \nu_\beta)^2 \ge \frac{d-t}{2} \max\left(0,\nu_\beta - \frac{1}{\beta \sqrt{d-t+1}}\right)^2  . $$
By the bounds developed in Section~\ref{subsec:proof-SC}, $\nu_\beta \ge \nu \ge \frac{2}{\sqrt{5}(d+2)^{3/2}} \mathrm{diam}(S)$. Hence
\begin{align*}
    f_\beta(x_t) - f_\beta(x_\star) &\geq \frac{d-t}{2} \max\left(0,\frac{2(\mathrm{diam}(S_\beta) -  2/\beta)}{\sqrt{5}(d+2)^{3/2}} - \frac{1}{\beta \sqrt{d-t+1}}\right)^2 \\
    &\geq \max\left(0,\sqrt{\frac{2}{5}\frac{(d-t)(\mathrm{diam}(S_\beta) -  2/\beta)^2}{(d+2)^3}} - \frac{1}{\sqrt{2}\beta}\right)^2  .
\end{align*}
Specializing to $t=T$ and $d=2(T+1)$, this provides a lower bound of
$$ f_\beta(x_T) - f_\beta(x_\star) \geq \max\left(0,\frac{\mathrm{diam}(S_\beta) -  2/\beta}{\sqrt{20}(T+2)} - \frac{1}{\sqrt{2}\beta }\right)^2  . $$

{\color{blue}
\section{Conclusion}
We have shown that the classical zero-chain lower bounds in the style of Nemirovski and Yudin can be extended to LMO methods over strongly convex sets. From our new adversarial oracle constructions, we provide a $\Omega(L\,\mathrm{dist}(S)^2/T^2)$ lower bound for $\alpha$-strongly convex problem instances with diameter $\Theta(1/(\alpha T))$. This only partially matches the rates of Garber and Hazan~\cite{garber2015faster}, differing in terms depending on the curvature $\alpha$ by a factor of $1/T^2$. Identification of an optimal algorithm and matching lower bound is left open. The PEP framework~\cite{drori2014performance,Interpolation,Interpolation2,luner2024performance} may be useful to this end, especially in tightening constants in bounds.

The fact that our theory requires $\mathrm{diam}(S)$ to be set as a function of the parameters $L, \alpha, \varepsilon$ is a weakness. While a general selection of $\mathrm{diam}(S)$ cannot be allowed (see the discussion following Theorem~\ref{thm:main}), more nuanced theory is needed. Likewise, our secondary results above for smooth sets were limited to the regime of ``modestly smooth sets''. Guarantees capturing the effect of larger set smoothness parameters are also of interest.

Finally, we note that the feasibility requirement of our lower bounding theory may be relaxed. Future work could characterize performance (upper and lower bounds) in terms of $f(x_T) - f(x_\star) + \lambda\mathrm{dist}(x_T,S)$ for an appropriate multiplier $\lambda$. Similarly, one could bound tradeoffs between suboptimality $f(x_T) - f(x_\star)$ and infeasibility $\mathrm{dist}(x_T,S)$ convergence. In settings where constraints are not absolute, this may be interesting.

}

    \paragraph{Acknowledgments.} Benjamin Grimmer was supported as an Alfred P. Sloan Foundation fellow.

    {\small
    \bibliographystyle{unsrt}
    \bibliography{bibliography}
    }
\end{document}